\documentclass[a4paper,11pt]{article}

\usepackage[utf8]{inputenc} 
\usepackage{amsfonts}
\usepackage{amssymb}
\usepackage{amsthm}
\usepackage{amsmath}
\usepackage{a4wide}
\usepackage{comment}
\usepackage{mathrsfs}
\usepackage{epsfig}
\usepackage{esint}
\usepackage{tikz}
\usepackage{fp,ifthen}
\usepackage{nicefrac}
\usetikzlibrary{decorations.pathreplacing}
\usepackage{float}
\usepackage{enumerate} 
\usepackage{enumitem} 

\usepackage[a4paper,twoside,top=2.5cm, bottom=2cm, left=2.3cm, right=2.3cm]{geometry} 

 \usepackage{lipsum}

\newcommand{\pdfgraphics}{\ifpdf\DeclareGraphicsExtensions{.pdf,.jpg}\else\fi}
\usepackage{graphicx}
\usepackage{dsfont}

\usepackage{color}
\definecolor{hanblue}{rgb}{0.27, 0.42, 0.81}
\definecolor{red}{rgb}{1.0, 0.0, 0.0}
\usepackage[colorlinks, citecolor=blue,linkcolor=blue, urlcolor = blue]{hyperref}
\usepackage[english,capitalize]{cleveref}

\usepackage{mathtools}

\theoremstyle{plain}

\newtheorem{teo}{Theorem}[section]
\newtheorem{lemma}[teo]{Lemma}
\newtheorem{prop}[teo]{Proposition}
\newtheorem{cor}[teo]{Corollary}
\newtheorem{ackn}{Acknowledgements\!}
\newtheorem{fndg}{Funding\!}

\theoremstyle{definition}
\newtheorem{defn}[teo]{Definition}

\theoremstyle{remark}
\newtheorem{rem}[teo]{Remark}

\numberwithin{equation}{section}

\newcommand{\de}{\,\ensuremath{\mathrm d}} 

\renewcommand{\epsilon}{\varepsilon}
\newcommand{\N}{\ensuremath{\mathbb N}}
\newcommand{\R}{\ensuremath{\mathbb R}}

\newcommand{\leb}{\Ell}

\newcommand{\Ell}{\mathscr{L}}

\makeatletter
\renewcommand*\env@matrix[1][*\c@MaxMatrixCols c]{%
  \hskip -\arraycolsep
  \let\@ifnextchar\new@ifnextchar
  \array{#1}}
\makeatother

\begin{document}

\pdfgraphics 

\title{Fourier transform of BV functions and applications}

\author{Thomas Beretti \footnote{International School for Advanced Studies (SISSA) - \href{thomas.beretti@sissa.it}{thomas.beretti@sissa.it}} \and Luca Gennaioli \footnote{University of Warwick, Department of Mathematics - \href{luca.gennaioli@warwick.ac.uk}{luca.gennaioli@warwick.ac.uk}}}

\date{\today}

\maketitle

\begin{abstract}
\noindent 
This paper investigates the relation between the Fourier transform of {\rm BV} (bounded variation) functions and their jump sets. We introduce the notion of $L^2$-jump product and obtain a weighted Plancherel identity for {\rm BV} functions. As a corollary, we get a newfound characterization of sets of finite perimeter in terms of their Fourier transform. Moreover, we sharpen a result of Herz on the set-theoretic derivative of the Fourier transform of characteristic functions of sets. Last, we obtain sharp bounds on the quadratic discrepancy of {\rm BV} functions, and as a consequence, we generalize the classic estimates of Beck and Montgomery.
\end{abstract}

\tableofcontents

\section{Introduction}

This work presents a novel bridge between harmonic analysis and geometric measure theory. The Fourier transform has proved to be an effective tool when studying differential equations. As well known, it allows linear differential problems to be turned into algebraic ones. Moreover, since Sobolev norms are equivalent to weighted $L^2$-norms of the Fourier transform, information on the latter can be exploited for studying the regularity of solutions of {\rm PDE}s. Another example of such duality is the relation between the Fourier restriction problem and the decay of solutions of some wave problems, such as the Schr\"odinger problem or the Klein-Gordon equations, and we refer the reader to the seminal paper of Strichartz~\cite{Str77}. \par

In Section~\ref{Fourier Symp}, we show that, for functions of bounded variation, the problem of estimating the short-time behaviour of the (relative) heat content is dual to determine the limit behaviour of averages of the Fourier transform. Since the former retrieves several quantities of geometric interest and is strictly related to perimeters, the latter will naturally inherit such a property. In general, the problem of estimating the Fourier transform has a long history, and we refer the reader to the classic work of Hlawka~\cite{Hla50}, Herz~\cite{Herz62}, and to the study in \cite{BNW88}.\par

 In what follows, $d$ stands for a positive integer, and for a positive value $R$, we write $B_R$ to denote the ball of radius $R$ centred in the origin. The Fourier transform of a function $u$ is either denoted by $\mathcal{F}\{u\}$ or $\hat u$.\par 
 We turn to the main result of Section~\ref{Fourier Symp}. Before proceeding with its statement, we introduce a special geometric quantity. In particular, the reader may find relevant geometric definitions and notation on {\rm BV} functions at the start of the same section.

\begin{defn}[$L^2$-jump product] Let $u,v\in {\rm BV}(\R^d)\cap L^\infty(\R^d)$, then we define the $L^2$-jump product of $u$ and $v$ as
\begin{equation*}
    \mathcal{J}(u,v)=\int_{J_u\cap J_v}(u^{+}-u^{-})(v^{+}-v^{-})\,\nu_{J_u}\cdot\nu_{J_v}\de\mathcal{H}^{d-1},
\end{equation*}
with the convention that $\mathcal{J}(u)=\mathcal{J}(u,u)$.
\end{defn}

We proceed with stating what we interpret as a weighted Plancherel identity.
\begin{teo}
\label{Supreme main theorem}
    Let $u,v\in{\rm BV}(\R^d)\cap L^\infty(\R^d)$, then it holds
    \begin{equation}
        \label{Fourier transform of BV functions 1}
        \lim_{R\to+\infty}\frac{2\pi^2}{R}\int_{B_R}|\xi|^2\hat{u}(\xi)\overline{\hat{v}}(\xi)\de\xi = \mathcal{J}(u,v).
    \end{equation}
   In particular, if $u=v$, we get 
    \begin{equation}
        \label{Fourier transform of BV functions 2}
        \lim_{R\to+\infty}\frac{2\pi^2}{R}\int_{B_R}|\xi|^2|\hat{u}|^2(\xi)\de\xi = \mathcal{J}(u).
    \end{equation}
\end{teo}
We make a few comments on the latter theorem. By classic properties of the Fourier transform, we can rewrite \eqref{Fourier transform of BV functions 2} as
\begin{equation}
    \label{Averaged Plancherel for measures}
     \lim_{R\to+\infty}\frac{1}{2R}\int_{B_R}|\mathcal{F}\{Du\}|^2(\xi)\de\xi = \mathcal{J}(u),
\end{equation}
where $Du$ stands for the distributional gradient of $u$. As a well-known property, the distributional gradient of a {\rm BV} function decomposes in an absolutely continuous part, a Cantor part, and a jump part. In \eqref{Averaged Plancherel for measures}, the jump part alone plays a role in the limit, and we give the following heuristic interpretation. The Fourier transform tells how influential a certain oscillation is, and when looking at increasingly high frequencies, one expects to detect faster oscillations. Since the oscillation frequency is “infinity” at the jump set of a function, then the jump part of the gradient is the only one caught up in the limit. \par
We also mention that we recover Wiener's lemma by considering $d=1$ in \eqref{Averaged Plancherel for measures}. Indeed, every finite measure in $\mathbb{R}$ is also the gradient of a (bounded) function of bounded variation. \par
Moreover, it is worth comparing Theorem~\ref{Supreme main theorem} with the results in \cite{Str90}, where analogous estimates for real-valued measures are obtained under some additional hypotheses, such as requiring a uniform control on the measure of balls. In particular, the techniques employed differ from ours, and we indirectly address the question of removing the regularity assumptions in the case of measures that are also distributional gradients. For completeness, we also mention the study in \cite{AgHo76} concerning measures supported on $\mathcal{C}^1$ submanifolds of $\R^d$, where the authors obtain similar estimates.\par

Integrating by parts \eqref{Fourier transform of BV functions 2}, we immediately deduce this first corollary to Theorem~\ref{Supreme main theorem}.
\begin{cor}
\label{Error in Plancherel corollary}
Let $u\in {\rm BV}(\R^d)\cap L^\infty(\R^d)$, then
    \begin{equation}
        \label{Intro result 2}
        \lim_{R\to+\infty}2\pi^2R\int_{B_R^\mathsf{c}}|\hat{u}|^2(\xi)\de\xi=\mathcal{J}(u).
    \end{equation}
\end{cor}
The latter equality is significant for at least two reasons, the first being that it quantifies the error committed by cutting off the higher frequencies in Plancherel identity. Namely, we obtain
\begin{equation*}
    \int_{\R^d}|{u}|^2(\xi)\de\xi=\int_{B_R}|\hat{u}|^2(\xi)\de\xi+\frac{\mathcal{J}(u)}{2\pi^2 R}+o(R^{-1})\quad\text{as}\quad R\to+\infty.
\end{equation*}
Moreover, by letting $\Omega\subset\mathbb{R}^d$ be a set of finite measure and perimeter, and by setting $u=\mathds{1}_{\Omega}$ in equation \eqref{Intro result 2}, we get
\begin{equation*}
\lim_{R\to+\infty}2\pi^2R\int_{B_R^\mathsf{c}}|\widehat{\mathds{1}}_\Omega|^2(\xi)\de\xi={\rm Per}(\Omega),
\end{equation*}
which extends the estimates in \cite[Ch.~6]{Mon94} to any dimension and for any set of finite perimeter.\par
Dealing with characteristic functions of sets, we further obtain this second corollary to Theorem~\ref{Supreme main theorem}. Again, we refer to the preliminary part of Section~\ref{Fourier Symp} for the notion of reduced boundary.
\begin{cor}
\label{Main corollary 1}
    Let $\Omega_1$ and $\Omega_2$ be sets of finite measure and perimeter of $\mathbb{R}^d$, then it holds
    \begin{equation}
        \label{Intro result mixed}
        \lim_{R\to+\infty}\frac{2\pi^2}{R}\int_{B_R}|\xi|^2\widehat{\mathds{1}}_{\Omega_1}(\xi)\overline{\widehat{\mathds{1}}}_{\Omega_2}(\xi)\de\xi = \underset{\partial^{*}\Omega_1\cap\partial^{*}\Omega_2}{\int}\nu_{\partial^{*}\Omega_1}\cdot\nu_{\partial^{*}\Omega_2}\de\mathcal{H}^{d-1}.
    \end{equation}
    In particular, if $\Omega_1=\Omega_2=\Omega$, we get
    \begin{equation}
    \label{Intro Result}
        \lim_{R\to+\infty}\frac{2\pi^2}{R}\int_{B_R}|\xi|^2|\widehat{\mathds{1}}_\Omega|^2(\xi)\de\xi = {\rm Per}(\Omega).
    \end{equation}
    Finally, if for a set $\Omega\subset\mathbb{R}^d$ of finite measure it holds
    \begin{equation}
        \label{Charact. of sets of finite perimeter}
         \limsup_{R\to+\infty}\frac{1}{R}\int_{B_R}|\xi|^2|\widehat{\mathds{1}}_\Omega|^2(\xi)\de\xi<+\infty,
    \end{equation}
    then $\Omega$ is of finite perimeter.
\end{cor}
The latter is a newfound characterization of sets of finite perimeter that does not involve distributional calculus and derivatives. Additionally, \eqref{Intro Result} improves on the estimates in \cite[Thm.~2]{Herz62}, where an upper bound for convex sets is obtained.\par

In Section~\ref{Fourier Sets}, again regarding sets of finite perimeter, we address another question of geometric measure theory and harmonic analysis. Namely, we characterize the set-theoretic derivative of the Fourier transform of sets of finite perimeter. The following result improves on \cite[Thm.~1]{Herz62}, in which the author only considered convex sets.
\begin{teo}
\label{Main theorem 2}
    Let $\Omega\subset\R^d$ be a set of finite measure and perimeter, then
    \begin{equation}
    \label{Derivative Fourier transform}
        \lim_{h\to 0}\frac{\widehat{\mathds{1}}_{\Omega+hB_1}(\xi)-\widehat{\mathds{1}}_\Omega(\xi)}{h}=\mathcal{F}\{{|D\mathds{1}_\Omega|}\}(\xi)\quad\forall\xi\in\R^d
    \end{equation}
    if and only if ${\rm Per}(\Omega)=\mathcal{S}\mathcal{M}(\Omega)$.
\end{teo}
The hypotheses under which we obtain \eqref{Derivative Fourier transform} are the weakest possible, and as an example, any set with Lipschitz boundary satisfies our hypotheses, as the reader may verify in \cite{ACV08}. In addition, our proof appears more straightforward than the one in \cite[Thm.~1]{Herz62} since it relies on classic functional analysis techniques and geometric measure theory.

In Section~\ref{Fourier Discrepancy}, we show how the asymptotic relation in \eqref{Fourier transform of BV functions 2} is key for improving on a major result in the theory of irregularities of distribution, also known as discrepancy theory. In what follows, we assume the dimension $d$ to be greater than or equal to $2$.\par 
For a set $\Omega\subset\mathbb{T}^d$ and for a set of $N$ points $\mathcal{P}_N\subset\mathbb{T}^d$, the discrepancy of $\mathcal{P}_N$ with respect to $\Omega$ refers to the quantity
\begin{equation*}
    \mathcal{D}(\Omega;\mathcal{P}_N)=\sum_{{p}\in\mathcal{P}_N}\mathds{1}_{\Omega}({p})-N|\Omega|.
\end{equation*}
The field had a pioneering advancement due to the seminal paper of Roth~\cite{Roth54}, where a new geometric point of view was introduced, leading to a broader employment of harmonic analysis. A later significant development followed by the work of Beck~\cite{Beck87} concerning discrepancy over affine transformations, and it turned out that such a problem is strictly related to estimating Fourier transforms. Before proceeding with our contribution, we need to introduce adequate notation.\par
Consider a real function $u\in L^1(\mathbb{R}^d)$, and let $\tau\in\mathbb{R}^d$ be a translation factor, let $\delta\geq0$ be a dilation factor, and let $\rho\colon\mathbb{R}^d\to\mathbb{R}^d$ be a rotation; in particular, we identify ${\rm SO}(d)$ with the set of all rotations. We define the affine transformation
\begin{equation*}
    [\tau,\delta,\rho]u({x})=u\left(\tau+\rho{x}/\delta\right).
\end{equation*}
Moreover, we consider the periodization functional ${\mathfrak{P}}\colon L^1(\mathbb{R}^d)\to L^1(\mathbb{T}^d)$ defined in the sense that
\begin{equation*}
    {\mathfrak{P}}\{u\}({x})=\sum_{{n}\in\mathbb{Z}^d}u({x}+{n}).
\end{equation*}
Then, for a function $u\in L^1(\mathbb{R}^d)$ and for a set of $N$ points $\mathcal{P}_N\subset\mathbb{T}^d$, we extend the previous notion of discrepancy by defining
\begin{equation}\label{Dis1}
    \mathcal{D}(u;\mathcal{P}_N)=\sum_{{p}\in\mathcal{P}_N}{\mathfrak{P}}\{u\}({p})-N\hat{u}(\mathbf{0}).
\end{equation}
In this setting, we define the quadratic discrepancy over affine transformations as
\begin{equation}\label{Dis2}
    \mathcal{D}_2(u;\mathcal{P}_N)=\int_{{\rm SO}(d)}\int_{0}^{1}\int_{\mathbb{T}^d}\left|\mathcal{D}([\tau,\delta,\rho]u; \mathcal{P}_N)\right|^2\de \tau \de \delta \de \rho.
\end{equation}\par
It is time to state the sharp bounds we obtain, starting from the upper one.
\begin{teo}\label{Upper}
There exists a positive dimensional constant $C_d$ such that for every $u\in{\rm BV}(\mathbb{R}^d)\cap L^\infty(\R^d)$, it holds
    \begin{equation*}
    \limsup_{N\to+\infty} N^{(1-d)/d}\inf_{\mathcal{P}_N\subset\mathbb{T}^d}\mathcal{D}_2(u;\mathcal{P}_N)\leq C_d\,  \mathcal{J}(u).
    \end{equation*}
\end{teo}
As the following result shows, the latter estimate is the best achievable with respect to the order of $N$ and the dependence on $u$.
\begin{teo}\label{Lower}
There exists a positive dimensional constant $C_d$ such that for every $u\in{\rm BV}(\mathbb{R}^d)\cap L^\infty(\R^d)$, it holds
    \begin{equation*}
    \liminf_{N\to+\infty} N^{(1-d)/d}\inf_{\mathcal{P}_N\subset\mathbb{T}^d}\mathcal{D}_2(u;\mathcal{P}_N)\geq C_d\,  \mathcal{J}(u).
    \end{equation*}
\end{teo}
In \cite{Beck87}, the latter lower bound is obtained for characteristic functions of convex compact sets (with non-empty interior) of $\mathbb{R}^d$, and in the same setting, the study in \cite{BC90} proved the order $N^{(1-d)/d}$ to be sharp, although a generic set-depending constant replaces the dependence on the perimeter. In an independent work, Montgomery~\cite[Ch.~6]{Mon94} showed that the same lower bound holds for every compact set (with non-empty interior) of $\mathbb{R}^2$ with piecewise-$\mathcal{C}^1$ boundary, hence trading convexity for a requirement of regularity. In particular, his work employed Fourier series and a refinement of an argument of Cassels~\cite{Cas56} for estimating sums of complex exponentials, and we briefly treat the latter in the last section.\par
Last, in the appendix, we adapt an argument of \cite{Herz62} and obtain a functional inequality that retrieves a non-sharp version of the celebrated isoperimetric inequality (when specialized to sets).

\begin{ackn}
    We thank Luca Brandolini, Camillo Brena, Michele Caselli, Leonardo Colzani, Davide Donati,  Giacomo Gigante, Antonio Pedro Ramos, Giorgio Stefani, and Giancarlo Travaglini for the valuable feedback and discussions.
\end{ackn}

\begin{fndg}
    L.G. is supported by UK Research and Innovation (UKRI) under the Horizon Europe funding guarantee [grant number EP/Z000297/1].
\end{fndg}

\section{Fourier asymptotics and {\rm BV} functions}
\label{Fourier Symp}
We shall introduce the main tools from the theory of {\rm PDE}s, harmonic analysis and geometric measure theory. The interested reader can consult \cite{Ev98} and \cite{Hor63} for a better account on {\rm PDE}s and harmonic analysis. For what concerns geometric measure theory, we refer the reader to \cite{AFP00} and \cite{M12}.\par

First, for a function $u\in L^1(\mathbb{R}^d)$, we define its Fourier transform as
\begin{equation*}
    \mathcal{F}\{u\}(\xi)=\hat{u}(\xi)=\int_{\R^d}u(x)e^{-2\pi i \xi\cdot x}\de x,
\end{equation*}
\newline
with clear extension to the space of tempered distributions $\mathcal{S}^\prime(\R^d)$. This choice of normalization implies the following version of Plancherel theorem: for $u,v\in L^2(\R^d)$, it holds
\begin{equation*}
    \int_{\R^d}u(x)v(x)\de x = \int_{\R^d}\hat{u}(\xi)\overline{\hat{v}}(\xi)\de\xi,
\end{equation*}
with $\overline{\hat{v}}$ being the complex conjugate of $\hat{v}$.\par

Now, consider a vector-valued measure $\Vec{\mu}=(\mu_1,...,\mu_d)$ in $\R^d$. We define the total variation of $\Vec{\mu}$ as the measure $|\Vec{\mu}|$ which, to any Borel set $E$, assigns the value
\begin{equation*}
    |\Vec{\mu}|(E)=\sup\left\{\sum_{n}\|\mu(E_n)\|_{\R^d}\,\colon\,(E_n)_n\text{ is a Borel partition of }E\right\}.
\end{equation*}
We say that a (possibly signed and vector-valued) measure $\nu$ is finite if
\begin{equation*}
    |\nu|(\R^d)<+\infty,
\end{equation*}
and we shall denote the vector space of finite measures with values in $\R^k$ with $\mathcal{M}(\R^d;\R^k)$, with the convention that $\mathcal{M}(\R^d)=\mathcal{M}(\R^d;\R^1)$. Equipping such spaces with the norm \begin{equation*}
    \|\mu\|_{\rm TV}=|\mu|(\R^d)
\end{equation*}
makes them Banach spaces. We also remark that it makes sense to compute the Fourier transform of a measure, and indeed, the Fourier transform \begin{equation*}
    \mathcal{F}:\mathcal{M}(\R^d)\to \mathcal{C}_b(\R^d)
\end{equation*}
is a linear and continuous operator with norm equal to $1$, and the same holds component-wise if we consider the operator \begin{equation*}
    \mathcal{F}:\mathcal{M}(\R^d;\R^d)\to \mathcal{C}_b(\R^d;\R^d).
\end{equation*}\par

Now, we proceed with defining the class of functions of bounded variation.
\begin{defn}[${\rm BV}$ functions]
We say that $u\in L^1(\R^d)$ is a function of bounded variation if $Du\in\mathcal{M}(\R^d;\R^d)$, where $Du$ is the distributional gradient of $u$. We denote the space of such functions with ${\rm BV}(\R^d)$.
\end{defn}

For a Borel set $E$ and a Borel measure $\mu$, we shall write $\mu_{|E}$ to denote the Borel measure such that $\mu_{|E}(A)=\mu(E\cap A)$ for any Borel set $A$. Now, we shall recall the decomposition of the gradient of a ${\rm BV}$ function as in \cite[Sec.~3.7]{AFP00}. Indeed, if $u\in{\rm BV}(\R^d)$, then $Du$ can be written as 
\begin{equation}
    \label{Decomposition of gradient}
    Du=\nabla u\leb^n+(u^{\vee}-u^{\wedge})\nu_{S_u}\mathcal{H}^{d-1}_{|S_u}+D^cu.
\end{equation}
In the previous decomposition, $\nabla u$ is the density of the absolutely continuous part with respect to $\leb^n$, and $D^cu$ is the Cantor part of the distributional gradient. The second term, often denoted with $D^ju$, is the jump part of the gradient, and we briefly explain the notation used for it. Here, $S_u$ denotes the approximate discontinuity set of $u$, while $\nu_{S_u}$ stands for its measure-theoretic normal. Namely, we have that $x\in S_u^\mathsf{c}$ if there exists $a\in\R$ such that 
\begin{equation}
    \label{Approximate continuity}
    \lim_{\rho\to 0^+}\fint_{B_\rho(x)}|u(y)-a|\de y=0,
\end{equation}
and on the other hand, for $x\in S_u$, we have 
\begin{equation*}
    \begin{split}
    u^{\vee}(x)&=\inf\big\{t\in[-\infty,+\infty]:\;\lim_{\rho\to 0^+}\rho^{-d}|\{u>t\}\cap B_\rho(x)|=0\big\}, \\
    u^{\wedge}(x)&=\sup\big\{t\in[-\infty,+\infty]:\;\lim_{\rho\to 0^+}\rho^{-d}|\{u<t\}\cap B_\rho(x)|=0\big\}.
\end{split}
\end{equation*}
We shall use the fact that there exists a $\mathcal{H}^{d-1}$-rectifiable Borel set $J_u\subseteq S_u$, usually referred as jump set of $u$, such that $\mathcal{H}^{d-1}(S_u\setminus J_u)=0$ and that allows us to rewrite \eqref{Decomposition of gradient} as
\begin{equation}
    \label{Gradient decomposition on the jump}
    Du=\nabla u\leb^d+(u^+-u^-)\nu_{J_u}\mathcal{H}^{d-1}_{J_u}+D^cu.
\end{equation}
Here, $u^+$ and $u^-$ are the one-sided approximate limits of $u$, defined in such a way that
\begin{equation*}
    \lim_{\rho\to 0^+}\fint_{B_{\rho}^+(x)}|u(y)-u^+(x)|\de y= \lim_{\rho\to 0^+}\fint_{B_{\rho}^-(x)}|u(y)-u^-(x)|\de y=0,
\end{equation*}
where, denoting by $\langle\cdot,\cdot\rangle$ the inner product of $\mathbb{R}^d$, we have
\begin{equation*}
    B_\rho^+(x)=\{x\in B_\rho(x):\langle x,\nu_{J_u}\rangle>0\}\quad\text{and}\quad B_\rho^-(x)=\{x\in B_\rho(x):\langle x,\nu_{J_u}\rangle<0\}.
\end{equation*}\par 

We proceed with the definition of sets of finite perimeter.
\begin{defn}[Finite perimeter set]
    For a Borel set $\Omega\subset\R^d$, we define the perimeter of $\Omega$ as 
    \begin{equation*}
        {\rm Per}(\Omega)=|D\mathds{1}_\Omega|(\R^d),
    \end{equation*}
    and we say that $\Omega$ is of finite perimeter if $\mathds{1}_\Omega\in{\rm BV}(\R^d)$.
\end{defn}
Due to the isoperimetric inequality, if $\Omega$ is of finite perimeter, then either $\Omega$ or $\Omega^\mathsf{c}$ is of finite measure. In the sequel, we shall assume that our sets of finite perimeter come with finite measure.\par
For a set of finite perimeter, it is possible to define the notion of reduced boundary, which identifies the more “regular” points of the boundary. In particular, this coincides with the jump set of the characteristic function of the set itself.
\begin{defn}[Reduced boundary]
    For a set of finite perimeter $\Omega\subset\mathbb{R}^d$, we define its reduced boundary $\partial^{*}\Omega$ as the Borel set of all the points $x\in{\rm supp}|D\mathds{1}_\Omega|$ such that the limit
    \begin{equation*}
        \nu_{\Omega}(x)=\lim_{r\to 0^+}\frac{D\mathds{1}_\Omega(B_{r}(x))}{|D\mathds{1}_\Omega|(B_{r}(x))}
    \end{equation*}
    exists and is equal to $1$.
\end{defn}
\begin{rem}
    Observe that $\partial^{*}\Omega$ is always a subset of the topological boundary, and it may happen that the two do not coincide. Trivially, the reduced boundary of a square consists of its topological boundary except for the $4$ points at the corners.
\end{rem}

We are now in the position to recall part of the celebrated structural result for sets of finite perimeter, namely De Giorgi theorem (see \cite[Thm.~3.59]{AFP00}).
\begin{teo}[De Giorgi]
Let $\Omega\subset\R^d$ be a set of finite perimeter, then $\partial^{*}\Omega$ is countably $\mathcal{H}^{d-1}$ rectifiable and $|D\mathds{1}_{\Omega}|=\mathcal{H}^{d-1}_{|\partial^{*}\Omega}$.
\end{teo}
\begin{rem}
    We do not delve into the definition of rectifiability in the above sense, and again, we refer the reader to \cite{AFP00} or \cite{M12}. However, one can think of it as a measure-theoretic notion of being $(d-1)$-dimensional.
\end{rem}

Now, following the notation in \cite{MPPP07}, for a function $u\in L^1$, we denote by $T(t)u$ the unique solution to the differential problem
\begin{equation}
\label{Heat equation}
\begin{cases}
        \partial_tv = \Delta v & \text{in }\R^d\times(0,+\infty) \\
        v(0,x)=u&\text{on }\R^d
\end{cases},
\end{equation}
evaluated at time $t>0$. Moreover, we recall the explicit representation via convolution
\begin{equation*}
    (T(t)u)(x)=\int_{\R^d}G_t^{\R^d}(x-y)u(y)\de y,
\end{equation*}
where $G_t^{\R^d}$ stands for the classic Gaussian heat kernel
\begin{equation}
    \label{Heat kernel}
    G_t^{\R^d}(z)=\frac{1}{(4\pi t)^{d/2}}e^{-\frac{|z|^2}{4t}}.
\end{equation}\par

Now, we recall two results on the (relative) heat content. First, we state \cite[Thm.~4.3]{MPPP07}.
\begin{teo}
\label{Differentiability Pallara}
    Let $u,v\in{\rm BV}(\R^d)\cap L^2(\R^d)$, then it holds
    \begin{equation}
        \label{heat convergence BV functions}
        \lim_{t\to 0^+}\frac{\sqrt{\pi}}{\sqrt{t}}\langle u-T(t)u,v\rangle = \int_{J_u\cap J_v}(u^+-u^-)(v^+-v^-)\nu_{J_u}\cdot\nu_{J_v}\de\mathcal{H}^{d-1}.
    \end{equation}
\end{teo}

\begin{rem}
For our purposes, given $u\in{\rm BV}(\R^d)\cap L^\infty(\mathbb{R}^d)$, we consider the map
\begin{equation}
\label{Relative BV heat content}
    H(t)=\int_{\R^d}(T(t^2)u)(x)v(x)\de x.
\end{equation}
By the latter result, we can only deduce that $H$ is (right) differentiable at $t=0$. In order to prove Theorem~\ref{Supreme main theorem}, we actually need
\begin{equation*}
    \lim_{t\to 0^+}H'(t)\quad\text{to exist},
\end{equation*}
and in this case, the limit would clearly coincide with the right-hand term of \eqref{heat convergence BV functions}. We address such a matter with Proposition~\ref{Technical proposition}.
\end{rem}


Secondly, we state \cite[Thm 3.4]{MPPP07}. The following is a sufficient condition for a set to be of finite perimeter in terms of its (relative) heat content, and we use it in the proof of Corollary~\ref{Main corollary 1}.
\begin{teo}
    \label{Pallara heat implies finite perimeter}
    Let $\Omega\subset\R^d$ be such that either $|\Omega|<+\infty$ or $|\Omega^\mathsf{c}|<+\infty$. If 
    \begin{equation*}
        \liminf_{t\to 0^+}\frac{1}{t}\int_{\Omega^\mathsf{c}}T(t^2)\mathds{1}_\Omega(x)\de x <+\infty,
    \end{equation*}
    then $\Omega$ has finite perimeter.
\end{teo}

 We proceed with the proof of the main technical result of the section.
 
\begin{prop}
\label{Technical proposition}
    Let $u,v\in {\rm BV}(\R^d)$, with $u\in L^\infty(\R^d)$ or $v\in L^\infty(\R^d)$, then 
    \begin{equation}
        \lim_{t\to 0^+}H^\prime(t)=-\frac{\mathcal{J}(u,v)}{\sqrt{\pi}}.
    \end{equation}
\end{prop}
\begin{proof}
    Without loss of generality, let us assume $u\in L^\infty(\R^d)$. Moreover, let $H$ be defined as in \eqref{Relative BV heat content}. Clearly $H\in \mathcal{C}^1((0,+\infty))$, and by \eqref{Heat equation}, for every $t\in(0,+\infty)$, we are allowed to write
    \begin{equation}
    \label{Derivative of H}
        H^\prime(t)=2t\int_{\R^d}\Delta (T(t^2)u)(x)v(x)\de x = -2t\int_{\R^d}\nabla (T(t^2)u)(x)\de Dv(x).
    \end{equation}
    We shall now compute the term $\nabla (T(t^2)u)(x)$ explicitly. Namely, we have
    \begin{equation*}
    \begin{split}
        \nabla (T(t^2)u)(x) &= \int_{\R^d}\nabla_xH_{\R^d}(x-y,t^2)u(y)\de y \\
        &=-\frac{1}{2t^2}\int_{\R^d}(x-y)\frac{e^{-\frac{|x-y|^2}{4t^2}}}{(4\pi t^2)^{d/2}}u(y)\de y \\
        &=-\frac{1}{(4\pi)^{d/2}2t}\int_{\R^d}ze^{-\frac{|z|^2}{4}}u(x-tz)\de z.
    \end{split}
\end{equation*}
    Therefore, \eqref{Derivative of H} becomes
    \begin{equation}
    \label{H derivative}
        H^\prime(t)=-\frac{1}{(4\pi)^{d/2}}\int_{\R^d}\int_{\R^d}ze^{-\frac{|z|^2}{4}}u(x-tz)\de z\de Dv(x).
    \end{equation}
    Now, let us consider
    \begin{equation}
    \label{g function}
        g(x,t)=\int_{\R^d}ze^{-\frac{|z|^2}{4}}u(x-tz)\de z,
    \end{equation}
    and split the integral in \eqref{H derivative} as
    \begin{equation*}
        H^\prime(t)=G_1(t)+G_2(t),
    \end{equation*}
    with
    \begin{equation*}
        G_1(t)\coloneq-\frac{1}{(4\pi)^{d/2}}\int_{J_u}g(x,t)\de Dv(x)\quad\text{and}\quad G_2(t)\coloneq-\frac{1}{(4\pi)^{d/2}}\int_{S_u^\mathsf{c}}g(x,t)\de Dv(x).
    \end{equation*}
    Here, as in \eqref{Decomposition of gradient}, $S_u$ and $J_u$ stand respectively for the approximate discontinuity set and jump set of $u$ (as in \eqref{Gradient decomposition on the jump}). In particular, we are allowed to write the latter decomposition since $\mathcal{H}^{d-1}(S_u\setminus J_u)=0$. Moreover, observe that $g(\cdot,t)$ is a continuous function such that
    \begin{equation}
    \label{Bound on g}
        \sup_{x\in \R^d}|g(x,t)|\leq \|u\|_{L^\infty}\int_{\R^d}|z|e^{-\frac{|z|^2}{4}}\de z= 2^{d+1}\Gamma\left((d+1)/2\right)\|u\|_{L^\infty}.
    \end{equation}
    \newline
    \emph{Claim}: $\lim_{t\to 0^+}G_2(t)=0$.
    \newline
    Let us prove the claim. Observe that, by \eqref{Approximate continuity}, for every $x\in S_u^\mathsf{c}$ we are allowed to write
    \begin{equation*}
        g(x,t)=\int_{\R^d}ze^{-\frac{|z|^2}{4}}\big(u(x-tz)-\Tilde{u}(x)\big)\de z,
    \end{equation*}
    where $\Tilde{u}(x)$ stands for the approximate limit of $u$ at $x$. Now, fix a $R>0$ and estimate $g$ as
    \begin{equation*}
    \begin{split}
        |g(x,t)|&\leq \int_{B_R}|z|e^{-\frac{|z|^2}{4}}|u(x-tz)-\Tilde{u}(x)|\de z + \int_{B_R^\mathsf{c}}|z|e^{-\frac{|z|^2}{4}}|u(x-tz)-\Tilde{u}(x)|\de z \\
        &\leq \int_{B_R}|z|e^{-\frac{|z|^2}{4}}|u(x-tz)-\Tilde{u}(x)|\de z +2\|u\|_{L^\infty}\int_{B_R^\mathsf{c}}|z|e^{-\frac{|z|^2}{4}}\de z.
    \end{split}
\end{equation*}
    We shall now take the $\limsup$ as $t\to 0^+$ of the previous expression, so that we get 
    \begin{equation*}
        \limsup_{t\to 0^+}|g(x,t)|\leq 2\|u\|_{L^\infty}\int_{B_R^\mathsf{c}}|z|e^{-\frac{|z|^2}{4}}\de z,
    \end{equation*}
    since $x$ is a point of approximate continuity. Finally, by taking the $\limsup$ as $R\to+\infty$, we obtain that
    \begin{equation*}
        \lim_{t\to 0^+}g(x,t)=0\quad\forall x\in S_u^\mathsf{c}. 
    \end{equation*}
    Using the fact that $|Dv|(\R^d)<+\infty$ and the bound in \eqref{Bound on g}, we obtain the claim by a dominated convergence argument.\par 
    
We are thus left to study what happens in the set $J_u$. First, we have that $J_u$ is countably $\mathcal{H}^{d-1}$-rectifiable, hence we can apply \cite[Prop.~3.92]{AFP00} to deduce that $|D^cv|(J_u)=0$. Moreover, the set $J_u$ has Lebesgue measure zero, and therefore, it only remains to deal with
\begin{equation}
\label{final boss}
   G_1(t)=-\frac{1}{(4\pi)^{d/2}}\int_{J_u\cap J_v}g(x,t)\cdot\nu_{J_v}(v^+-v^-)\de\mathcal{H}^{d-1}(x).
\end{equation}
Now, define
\begin{equation*}
    f(x,t)\coloneq\langle g(x,t),\nu_{J_v}(x)\rangle,
\end{equation*}
and consider 
\begin{equation*}
    H^{+}_{\nu_{J_u}(x)}=\{z\in\R^d:\;\langle z,\nu_{J_u}(x)\rangle\geq0\}\quad\text{and}\quad H^{-}_{\nu_{J_u}(x)}=\{z\in\R^d:\;\langle z,\nu_{J_u}(x)\rangle\leq0\}.
\end{equation*}
By exploiting the symmetries of the integrand in \eqref{final boss}, we get
\begin{equation*}
\begin{split}
    f(x,t)&=\int_{H_{\nu_{J_u}}^+}\langle z,\nu_{J_v}(x)\rangle e^{-\frac{|z|^2}{4}}(u(x-tz)-u^+(x))\de z +u^+(x)\int_{H_{\nu_{J_u}}^+}\langle z,\nu_{J_v}(x)\rangle e^{-\frac{|z|^2}{4}}\de z \\
    &+\int_{H_{\nu_{J_u}}^-}\langle z,\nu_{J_v}(x)\rangle e^{-\frac{|z|^2}{4}}(u(x-tz)-u^-(x))\de z +u^-(x)\int_{H_{\nu_{J_u}}^-}\langle z,\nu_{J_v}(x)\rangle e^{-\frac{|z|^2}{4}}\de z \\
    &=:A+B+C+D.
\end{split}
\end{equation*}
\newline
\emph{Claim:} $\lim_{t\to 0^+}A=\lim_{t\to 0^+}C=0$ for every $x\in J_u\cap J_v$.
\newline
First, let us consider the term $A$; in particular, the term $C$ can be treated analogously. We have 
\begin{equation*}
    \begin{split}
    A&=\int_{B_R\cap H_{\nu_{J_u}}^+}\langle z,\nu_{J_v}(x)\rangle e^{-\frac{|z|^2}{4}}(u(x-tz)-u^+(x))\de z\\
    &+\int_{B_R^\mathsf{c}\cap H_{\nu_{J_u}}^+}\langle z,\nu_{J_v}(x)\rangle e^{-\frac{|z|^2}{4}}(u(x-tz)-u^+(x))\de z.
\end{split}
\end{equation*}
By the definition of $u^+(x)$, the first integral on the right-hand side goes to zero as $t\to 0$. For the second integral, as in \eqref{Bound on g}, we have that for every $t\in (0,+\infty)$ it holds
\begin{equation*}
    \int_{B_R^\mathsf{c}\cap H_{\nu_{J_u}}^+}\langle z,\nu_{J_v}(x)\rangle e^{-\frac{|z|^2}{4}}(u(x-tz)-u^+(x))\de z\leq\|u\|_{L^\infty}\int_{B_R^\mathsf{c}}|z|e^{-|z|^2/4}\de z,
\end{equation*}
and therefore, it converges to $0$ as $R\to+\infty$. This proves the claim.\par

Now observe that, since for $\mathcal{H}^{d-1}$-a.e. $x\in J_u\cap J_v$ it holds
\begin{equation*}
    \nu_{J_v}(x)=\nu_{J_u}(x)\langle \nu_{J_u}(x),\nu_{J_v}(x)\rangle,
\end{equation*}
then we get
\begin{equation*}
    B= \frac{(4\pi)^{d/2}}{\sqrt{\pi}}u^+(x)\langle \nu_{J_u}(x),\nu_{J_v}(x)\rangle\quad\text{and}\quad D= -\frac{(4\pi)^{d/2}}{\sqrt{\pi}}u^-(x)\langle \nu_{J_u}(x),\nu_{J_v}(x)\rangle.
\end{equation*}
A straightforward application of dominated convergence finally gives 
\begin{equation*}
    \lim_{t\to 0^+}G_1(t)=-\frac{1}{\sqrt{\pi}}\int_{J_u\cap J_v}(u^+-u^-)(v^+-v^-)\nu_{J_u}\cdot\nu_{J_v}\de\mathcal{H}^{d-1},
\end{equation*}
which is the thesis.
\end{proof}


Before proceeding with the proof of Theorem~\ref{Supreme main theorem}, we recall the integral Abelian-Tauberian theorem. The reader may find related proofs in the appendix of \cite{AHT18}.
\begin{teo}[Abelian-Tauberian]
 \label{Abelian Tauberian}
    Let $\nu$ be a $\sigma$-finite positive Borel measure and let $\gamma\in[0,+\infty)$, then it holds
    \begin{equation*}
        \lim_{t\to 0^+}t^\gamma\int_0^{+\infty} e^{-t\lambda}\de\nu(\lambda)=C\qquad\iff\qquad \lim_{a\to+\infty}\frac{\nu([0,a])}{a^\gamma}=\frac{C}{\Gamma(\gamma+1)}.
    \end{equation*}
\end{teo}
In the proof of Corollary~\ref{Main corollary 1}, we use a special one-sided version of the latter theorem, in which the existence of the limit is not required. Again, we refer to the appendix of \cite{AHT18}.
\begin{teo}
\label{limsup Abelian Tauberian}
Let $\nu$ be a $\sigma$-finite positive Borel measure and let $\gamma\in[0,+\infty)$. If 
\begin{equation*}
    \limsup_{a\to\infty}\frac{\nu([0,a])}{a^\gamma}\leq C<+\infty
\end{equation*}
for some $C\geq 0$, then
\begin{equation*}
    \limsup_{t\to 0^+}t^{\gamma}\int_0^\infty e^{-tx}d\nu(x)\leq C\Gamma(\gamma+1).
\end{equation*}
\end{teo}
Finally, we have gathered all the tools to prove Theorem~\ref{Supreme main theorem}.

\begin{proof}[Proof of Theorem~\ref{Supreme main theorem}]
First, we shall prove \eqref{Fourier transform of BV functions 2}.
Let $u$ be as in the assumption, and similarly to \eqref{Relative BV heat content}, consider the quantity  
\begin{equation*}
    H(t)=\int_{\R^d}(T(t^2)u)(x)u(x)\de x.
\end{equation*}
    By Plancherel theorem and by the explicit expression of the heat kernel in \eqref{Heat kernel}, we get 
\begin{equation*}
    H(t) = \int_{\R^d}\mathcal{F}\{T(t^2)u\}(\xi)\overline{\mathcal{F}\{u\}}(\xi)\de\xi = \int_{\R^d}e^{-4\pi^2|\xi|^2t^2}|\hat{u}|^2(\xi)\de\xi.
\end{equation*}
Observe that $t\mapsto H(t)$ is continuous in $[0,1]$, and by Plancherel theorem and dominated convergence, it follows that $H(0)=\|u\|_{L^2}$. Again, by dominated convergence, we obtain
\begin{equation*}
    H^{\prime}(t)=-8\pi^2t\int_{\R^d}|\xi|^2e^{-4\pi^2|\xi|^2t^2}|\hat{u}|^2(\xi)\de\xi,
\end{equation*}
and in particular, $H^{\prime}$ is continuous on $(0,1]$. Therefore, thanks to Proposition~\ref{Technical proposition}, we get  
\begin{equation*}
    t\int_{\R^d}|\xi|^2e^{-4\pi^2|\xi|^2t^2}|\hat{u}|^2(\xi)\de\xi\to\frac{\mathcal{J}(u)}{8\pi^{\frac{5}{2}}}.
\end{equation*}
Now, let us convert to polar coordinates and obtain
\begin{equation*}
    t\int_0^{+\infty} e^{-4\pi^2r^2t^2}r^{n+1}g(r)\de r \to\frac{\mathcal{J}(u)}{8\pi^{\frac{5}{2}}},
\end{equation*}
where we have defined
\begin{equation*}
    g(r)\coloneq\int_{\mathbb{S}^{d-1}}|\hat{u}|^2(rv)\de\sigma(v).
\end{equation*}
Let us write $\delta \coloneq 4\pi^2t^2$, and then, by the change of variable $r^2=s$, we get
\begin{equation}
\label{almost done}
    \int_0^{+\infty} e^{-\delta s}s^{\frac{n}{2}}g(\sqrt{s})\de s\sim \frac{\mathcal{J}(u)}{2\pi^{\frac{3}{2}}\sqrt{\delta}}\quad\text{as}\quad s\to0^+,
\end{equation}
where we use the symbol “$\sim$” to denote the asymptotic behaviour. Now, consider
\begin{equation*}
    \Phi(s)\coloneq\int_0^s z^{\frac{n}{2}}g(\sqrt{z})\de z=2\int_0^{\sqrt{s}}y^{n+1}g(y)\de y = 2\int_{B_{\sqrt{s}}}|\xi|^2|\hat{u}|^2(\xi)\de\xi.
\end{equation*}
With the previous definition in mind, we can write \eqref{almost done} as 
\begin{equation*}
    \int_0^{+\infty} e^{-\delta s}\de \Phi(s)\sim\frac{\mathcal{J}(u)}{2\pi^{\frac{3}{2}}\sqrt{\delta}} \quad\text{as}\quad \delta\to0^+,
\end{equation*}
Finally, we apply Theorem~\ref{Abelian Tauberian} and obtain 
\begin{equation*}
    \Phi(R)\sim \sqrt{R}\frac{\mathcal{J}(u)}{\pi^2}\quad\text{as}\quad R\to+\infty,
\end{equation*}
which (by replacing $R$ with $R^2$) means that 
\begin{equation}
    \lim_{R\to+\infty}\frac{1}{R}\int_{B_R}|\xi|^2|\hat{u}|^2(\xi)\de\xi=\frac{\mathcal{J}(u)}{2\pi^2}.
\end{equation}\par
In order to prove \eqref{Fourier transform of BV functions 1}, we shall proceed by polarization. Again, let $u$ and $v$ satisfy our hypotheses. We exploit the fact that 
\begin{equation*}
    \Re\{\hat{u}(\xi)\overline{\hat{v}}(\xi)\}=\frac{|\widehat{(u+v)}(\xi)|^2}{2}-\frac{|\hat{u}(\xi)|^2}{2}-\frac{|\hat{v}(\xi)|^2}{2}.
\end{equation*}
Indeed, by the latter equation and by \eqref{Fourier transform of BV functions 2}, we obtain
\begin{equation}
\label{Final integrals}
    \begin{split}
     \lim_{R\to+\infty}\frac{1}{R}\int_{B_R}|\xi|^2\hat{u}(\xi)\overline{\hat{v}}(\xi)\de\xi&=\frac{1}{4\pi^2}\left( \int_{J_{u+v}}((u+v)^{+}-(u+v)^{-})^2\de\mathcal{H}^{d-1} \right.\\
     &\left.-\int_{J_u}(u^{+}-u^{-})^2\de\mathcal{H}^{d-1}-\int_{J_v}(v^{+}-v^{-})^2\de\mathcal{H}^{d-1}\right).
\end{split}
\end{equation}
Now, we write 
\begin{equation*}
    \begin{split}
    J_{u+v}&=(J_u\setminus J_v)\cup(J_v\setminus J_u)\cup(J_u\cap J_v), \\
    J_u&=(J_u\setminus J_v)\cup(J_u\cap J_v), \\
    J_v&=(J_v\setminus J_u)\cup(J_u\cap J_v),
\end{split}
\end{equation*}
 and we split the integrals at the right-hand side of \eqref{Final integrals} accordingly. In particular, we observe that if $x\in J_u\cap J_v$, then it holds
 \begin{equation*}
     (u+v)^{\underset{-}{+}}(x)=u^{\underset{-}{+}}(x)+v^{\underset{-}{+}}(x).
 \end{equation*}Finally, a simple rearrangement of the terms gives \eqref{Fourier transform of BV functions 1}.
\end{proof} 

Now, we proceed with the proof of the first corollary to Theorem~\ref{Supreme main theorem}.

\begin{proof}[Proof of Corollary~\ref{Error in Plancherel corollary}]
    Before performing the limit in \eqref{Fourier transform of BV functions 2}, multiply by $R^{-3}$ both sides and then integrate with respect to $R$ from $L>0$ to $+\infty$, so to get
\begin{equation*}
    \int_L^\infty R^{-3} \int_{B_R}|\xi|^2|\hat{u}|^2(\xi)\de\xi\sim \frac{\mathcal{J}(u)}{2\pi^2L}\quad\text{as}\quad L\to+\infty.
\end{equation*}
Now, we integrate by parts (justified by \eqref{Fourier transform of BV functions 2}) the left-hand side above and obtain
\begin{equation*}
    \frac{1}{2L^2}\int_{B_L}|\xi|^2|\hat{u}|^2(\xi)\de\xi+\frac{1}{2}\int_{B_L^\mathsf{c}}|\hat{u}|^2(\xi)\de\xi\sim  \frac{\mathcal{J}(u)}{2\pi^2L} \quad\text{as}\quad L\to+\infty.
\end{equation*}
Finally, by applying \eqref{Fourier transform of BV functions 2} a second time, we end up with 
\begin{equation*}
    \int_{B_L^\mathsf{c}}|\hat{u}|^2(\xi)\de\xi\sim  \frac{\mathcal{J}(u)}{2\pi^2L}\quad\text{as}\quad L\to+\infty,
\end{equation*}
which is indeed \eqref{Intro result 2}.
\end{proof}

We proceed with the proof of the second and last corollary, in which we present a newfound characterization of sets of finite perimeter.

\begin{proof}[Proof of Corollary~\ref{Main corollary 1}]
The relations \eqref{Intro result mixed} and \eqref{Intro Result} easily follow from Theorem~\ref{Supreme main theorem} and the definition of reduced boundary. Hence, let us suppose that \eqref{Charact. of sets of finite perimeter} holds. By Theorem \ref{limsup Abelian Tauberian}, we have
\begin{equation*}
    \limsup_{\delta\to 0^+}\sqrt{\delta}\int_0^{+\infty} e^{-\delta s}\de\Phi(s)<+\infty,
\end{equation*}
with (we shall neglect the constants)
\begin{equation*}
    \Phi(s)=\int_0^s z^{\frac{n}{2}}g(\sqrt{z})\de z\quad\text{and}\quad g(r)=\int_{\mathbb{S}^{d-1}}|\widehat{\mathds{1}}_\Omega|^2(rv)\de v.
\end{equation*}
By a change of variables, we obtain
\begin{equation*}
    \limsup_{t\to 0^+}t\int_{\R^d}e^{-|\xi|^2t^2}|\xi|^2|\widehat{\mathds{1}}_\Omega|^2(\xi)\de\xi<+\infty.
\end{equation*}
By Plancherel theorem, the latter inequality implies that
\begin{equation*}
    \limsup_{t\to 0^+}|H^\prime(t)|<+\infty,
\end{equation*}
where we intend $H^\prime$ as in \eqref{Derivative of H} with $u=v=\mathds{1}_\Omega$. However, by the mean value theorem and by the monotonicity properties of $H$, it follows that
\begin{equation}
\label{Diff at zero}
    0\leq\limsup_{t\to 0}\frac{H(0)-H(t)}{t}<+\infty.
\end{equation}
Exploiting the definition of $H$, we turn \eqref{Diff at zero} into
\begin{equation*}
    \limsup_{t\to 0^+}\frac{1}{t}\int_{\Omega^\mathsf{c}} \big(T(t^2)\mathds{1}_\Omega\big)(x)\de x <+\infty,
\end{equation*}
and we can conclude by Theorem~\ref{Pallara heat implies finite perimeter}. Last, the converse implication is included in Theorem~\ref{Supreme main theorem}.
\end{proof}

\section{Fourier transform of derivative of sets}
\label{Fourier Sets}
We shall write $\Omega_r$ to denote the set
\begin{equation*}
    \Omega_r=\{x\in\R^d:\;{\rm d}(\Omega,x)\leq r\},
\end{equation*}
which coincides with the Minkowski sum $\Omega+B_r$. We have the following definition.

\begin{defn}[outer Minkowski content]
\label{Outer Minkowski content}
We say that a Borel set $\Omega\subseteq\R^d$ has a finite outer Minkowski content if the limit
\begin{equation*}
    \mathcal{S}\mathcal{M}(\Omega)=\lim_{r\to 0^+}\frac{|\Omega_r\setminus\Omega|}{r}
\end{equation*}
exists and is finite.
\end{defn}

Moreover, we give the following notion of convergence for sequences of measures.
\begin{defn}
    We say that a sequence of measures $(\mu_n)_{n\in\N}\subset\mathcal{M}(\R^d)$ weakly converges to a measure $\mu\in\mathcal{M}(\R^d)$, if for every $\varphi\in\mathcal{C}_b(\R^d)$ it holds 
    \begin{equation*}
    \lim_{n\to+\infty}\int_{\R^d}\varphi\de\mu_n=\int_{\R^d}\varphi\de\mu.
    \end{equation*}
\end{defn}

The following result is essential to obtain Theorem~\ref{Main theorem 2}, and it is a refinement of \cite[Prop.~2]{ACV08}.

\begin{prop}
\label{Prop ambrosio refined}
    Let $\Omega\subset\R^d$ be a set of finite measure and perimeter. Then $\mathcal{S}\mathcal{M}(\Omega)={\rm Per}(\Omega)$ if and only if the measures 
    \begin{equation}
    \label{Measures mu}
        \mu_h\coloneq\frac{\mathds{1}_{\Omega+hB_1}-\mathds{1}_{\Omega}}{h}
    \end{equation}
    weakly converge to $|D\mathds{1}_\Omega|$.
\end{prop}
\begin{proof}
 First, let us assume that
 \begin{equation*}
     \mathcal{S}\mathcal{M}(\Omega)={\rm Per}(\Omega)<+\infty.
 \end{equation*}
 By Portmanteau theorem (see \cite[Thm.~8.2.3]{Bog07}), and since the assumption implies that
 \begin{equation*}
     \lim_{h\to 0^+}\mu_h(\R^d)={\rm Per}(\Omega),
 \end{equation*}
 then it is enough to show that
 \begin{equation*}
     \liminf_{h\to 0^+}\mu_h(U)\geq|D\mathds{1}_\Omega|(U)\quad\text{for every open set }U.
 \end{equation*}
 Therefore, consider an open set $U$ and define the distance function
 \begin{equation*}
     {\rm d}_\Omega(x)\coloneq\inf\{|x-y|:\;y\in\Omega\}.
 \end{equation*}
 It is clear that such a function is Lipschitz continuous, and as in \cite[p.~11]{AmDan00}, we have that
 \begin{equation*}
     |\nabla{\rm d}_\Omega|=1\quad \leb^d\text{-almost everywhere in }\Omega^\mathsf{c}.
 \end{equation*}
 Given the previous observations, we apply Coarea formula to obtain
 \begin{equation*}
     \mu_h(U)=\frac{1}{h}\underset{\{0<\de_{\Omega}<h\}}{\int}\mathds{1}_U|\nabla\de_\Omega|\de\mathcal{H}^d=\frac{1}{h}\int_0^h{\rm Per}(\{{\rm d}_\Omega>t\};U)\de t=\int_0^1{\rm Per}(\{{\rm d}_\Omega>hr\};U)\de r.
 \end{equation*}
 Since the sets $\{{\rm d}_\Omega>th\}$ converge locally in measure to $\{{\rm d}_\Omega>0\}=\Omega^\mathsf{c}$ as $h\to 0^+$, then, by the lower semicontinuity of the perimeter (see \cite[Prop.~3.38]{AFP00}) and by Fatou lemma, we get 
 \begin{equation*}
     \liminf_{h\to 0^+}\mu_h(U)\geq{\rm Per}(\Omega^\mathsf{c};U)={\rm Per}(\Omega;U)=|D\mathds{1}_\Omega|(U).
 \end{equation*}
 This proves the first part of the result.
 \par
 Conversely, let $\mu_h$ be weakly convergent to $|D\mathds{1}_\Omega|$. Then this implies the boundedness of the sequence, and moreover, we have that
 \begin{equation*}
     \mu_h(\R^d)\to|D\mathds{1}_\Omega|(\R^d)={\rm Per}(\Omega),
 \end{equation*}
 which gives the thesis. 
\end{proof}

In light of the latter proposition, we give a short proof of Theorem~\ref{Main theorem 2}.
\begin{proof}[Proof of Theorem~\ref{Main theorem 2}]

Notice that the function
\begin{equation*}
    x\mapsto e^{-2\pi i\xi\cdot x}
\end{equation*}
is bounded and continuous for all $\xi\in\R^d$. If the measures $\mu_h$ (defined as in \eqref{Measures mu}) converge in duality with $\mathcal{C}_b(\R^d)$ to $|D\mathds{1}_\Omega|$, then we have 
\begin{equation*}
    \lim_{h\to 0^+}\int_{\R^d}e^{-2\pi i\xi\cdot x}\de\mu_h(x)=\int_{\R^d}e^{-2\pi i\xi\cdot x}\de|D\mathds{1}_\Omega|(x),
\end{equation*}
which is equivalent to \eqref{Derivative Fourier transform}. The converse implication is trivial since one can test the pointwise convergence of the Fourier transforms at $\xi=0$, easily implying the thesis.
\end{proof}

\section{Discrepancy of {\rm BV} functions}
\label{Fourier Discrepancy}

We now turn our attention to the quadratic discrepancy and prove Theorem~\ref{Upper} and Theorem~\ref{Lower}. Let us start by exploiting the convolutional structure of \eqref{Dis1}. Namely, if we let $\mu_{\rm L}$ be the Lebesgue measure on $\mathbb{T}^d$, and for ${p}\in\mathbb{T}^d$ we consider $\mu_{\rm D}({p})$ to be the Dirac delta centered at ${p}$, then by setting
\begin{equation*}
    \tilde{\mu}=\sum_{{p}\in\mathcal{P}_N}\mu_{\rm D}(-{p})-N\mu_{\rm L},
\end{equation*}
for a function $u\in L^1(\mathbb{R}^d)$, we may rewrite \eqref{Dis1} as
\begin{equation*}
\mathcal{D}(u;\mathcal{P}_N)=\int_{\mathbb{T}^d}\mathfrak{P}\{u\}({x}) \de \tilde{\mu}(-{x})=(\mathfrak{P}\{u\}\ast\tilde{\mu})(\mathbf{0}).
\end{equation*}
Now, for $f\in L^1(\mathbb{T}^d)$ or $f\in\mathcal{M}(\mathbb{T}^d)$, let $\tilde{\mathcal{F}}\{f\}\colon \mathbb{Z}^d\to\mathbb{C}$ be the function of the Fourier coefficients of $f$. Then, it is not difficult to see that for $u\in L^1(\mathbb{R}^d)$, and for every ${n}\in\mathbb{Z}^d$, it holds
\begin{equation*}
    \tilde{\mathcal{F}}\circ \mathfrak{P}\{u\}({n})=\mathcal{F}\{u\}({n}).
\end{equation*}
Moreover, by classic properties of the Fourier transform, it holds
\begin{equation}\label{FourierProp}
    \mathcal{F}\left\{ [\mathbf{0},\delta,\rho]u\right\}=\delta^{d}\,[\mathbf{0},\delta^{-1},\rho^{-1}]\mathcal{F}\left\{u\right\}.
\end{equation}
Therefore, by applying Plancherel theorem on $\mathbb{T}^d$, we get
\begin{equation*}
\begin{split}
    \int_{\mathbb{T}^d}\left|\mathcal{D}([\tau,\delta,\rho]u; \mathcal{P}_N)\right|^2\de\tau&=\int_{\mathbb{T}^d}\left|(\mathfrak{P}\{[\mathbf{0},\delta,\rho]u\} \ast \tilde{\mu})\right|^2(\tau)\de\tau\\
    &=\sum_{{n}\in\mathbb{Z}^d}\left|\tilde{\mathcal{F}}\{\tilde{\mu}\}({n})\right|^2\left|\tilde{\mathcal{F}}\circ \mathfrak{P}\{[\mathbf{0},\delta,\rho]u\}({n})\right|^2\\
    &=\sum_{{n}\in\mathbb{Z}_*^d}\left|\sum_{{p}\in\mathcal{P}_N}e^{2 \pi i {p}\cdot{n}}\right|^2\left|\mathcal{F}\{[\mathbf{0},\delta,\rho]u\}({n})\right|^2,
\end{split}
\end{equation*}
where, for the sake of notation, we have set $\mathbb{Z}^d_*=\mathbb{Z}^d\setminus\{\mathbf{0}\}$. Finally, by \eqref{FourierProp} and the symmetry of ${\rm SO}(d)$, we may rewrite \eqref{Dis2} as
\begin{equation*}
    \mathcal{D}_2(u;\mathcal{P}_N)=\sum_{{n}\in\mathbb{Z}_*^d}\left|\sum_{{p}\in\mathcal{P}_N}e^{2 \pi i {p}\cdot{n}}\right|^2\int_{{\rm SO}(d)}\int_{0}^{1}\delta^{2d}\left|[\mathbf{0},\delta^{-1},\rho]\mathcal{F}\{u\}({n})\right|^2\de\delta\de\rho,
\end{equation*}
and further, by a change of variable, we get
\begin{equation}\label{Fourier Series}
\begin{split}
    \mathcal{D}_2(u;\mathcal{P}_N)= \sum_{{n}\in\mathbb{Z}_*^d}\left|\sum_{{p}\in\mathcal{P}_N}e^{2 \pi i {p}\cdot{n}}\right|^2|{n}|^{-2d-1}\int_{B_{|{n}|}}|\xi|^{d+1}|\hat{u}|^2(\xi)\de\xi.
\end{split}
\end{equation}
The rest of the section relies on the following asymptotic relation resulting from integrating by parts in \eqref{Fourier transform of BV functions 2}. Namely, it holds
\begin{equation}\label{Fundamental Asymptotic}
    \int_{B_{|{n}|}}|\xi|^{d+1}|\hat{u}|^2(\xi)\de\xi\sim\frac{{|{n}|^{d}\, \mathcal{J}}(u)}{2d\pi^2}\quad\text{as}\quad n\to\infty.
\end{equation}
Hence, let us start by proving the upper bound in Theorem~\ref{Upper}.

\begin{proof}[Proof of Theorem~\ref{Upper}]
    We write $c_i=c_i(d)$ with $i\in\mathbb{N}$ to denote generic dimensional constants throughout the proof. By \eqref{Fundamental Asymptotic}, there exists $n_0\in\mathbb{N}$ be such that for every ${n}\geq n_0$, it holds
\begin{equation*}
    \int_{B_{|{n}|}}|\xi|^{d+1}|\hat{u}|^2(\xi)\de\xi\leq\frac{{|{n}|^{d}\, \mathcal{J}}(u)}{d\pi^2}.
\end{equation*}
First, we prove the upper bound for all positive integers $N$ such that $N^{1/d}\geq n_0$ and $N^{1/d}\in\mathbb{N}$. Consider the set of $N$ points $\mathcal{P}_N\subset\mathbb{T}^d$ to be the periodization of
    \begin{equation}\label{Lattice}
    \frac{[0,N^{1/d})^d\cap\mathbb{Z}^d}{N^{1/d}}\subset[0,1)^d.
    \end{equation}
    Hence, it is not difficult to see that
    \begin{equation*}
        \sum_{{p}\in\mathcal{P}_N}e^{2 \pi i {p}\cdot{n}}=\begin{cases}
            N &{\rm if}\quad{n}\in N^{1/d}\,\mathbb{Z}^d\\
            0 &{\rm else}
        \end{cases},
    \end{equation*}
    and therefore, by \eqref{Fourier Series} and the assumptions on $N$, we get
    \begin{equation*}
        \begin{split}
            \mathcal{D}_2(u;\mathcal{P}_N)&=\sum_{{m}\in\mathbb{Z}_*^d}N^2|N^{1/d}\,{m}|^{-2d-1}\int_{|\xi|\leq N^{1/d}\,|{m}|}|\xi|^{d+1}|\hat{u}|^2(\xi)\de\xi\\
        &\leq \frac{1}{d\pi^2}\, \mathcal{J}(u)\,N^2 \sum_{{n}\in\mathbb{Z}_*^d}|N^{1/d}\,{m}|^{-d-1}\\
        &\leq c_1 \,  \mathcal{J}(u)\,N^{(d-1)/d}.
        \end{split}
    \end{equation*}\par

    Now, we prove the upper bound for every positive integer $N$. As $k$ goes through the positive integers, we consider the recursive decomposition of $N$ defined as
    \begin{equation*}
    N_k=\max\left\{n\in\mathbb{N}\,\colon\, n^{1/d}\in\mathbb{N},\: n\leq N-\sum_{j=1}^{k-1}N_j\right\},
    \end{equation*}
    where improper sums are zeros. By the latter definition, we get that
    \begin{equation*}
    N-\sum_{j=1}^{k-1}N_j \leq(N_k^{1/d}+1)^d\leq N_k+c_2N_k^{(d-1)/d},
    \end{equation*}
    and therefore, for every integer $k\geq1$, it holds
    \begin{equation}\label{Relazione}
        N_{k+1}\leq N-\sum_{j=1}^{k}N_j\leq c_2N_k^{(d-1)/d}.
    \end{equation}
    Hence, since by definition $N_1\leq N$, then we recursively get
    \begin{equation}\label{Recursione}
        N_{k+1}^{1/d}\leq c_2^{1/d}N_k^{(d-1)/d^2}\leq c_2^{(k-1)/d} N^{(d-1)^{k-1}/{d^k}}.
    \end{equation}
    Now, set $K=K(d)$ to be the smallest integer such that
    \begin{equation*}
        \frac{(d-1)^{K}}{d^{K}}\leq\frac{d-1}{4d},
    \end{equation*}
    and in particular, notice that by \eqref{Recursione}, we have
    \begin{equation}\label{Utile}
        N_{K+1}^{(d-1)/d}\leq c_3\,N^{(d-1)/(4d)}.
    \end{equation}
    Now, consider the remainder
    \begin{equation*}
        N_0=N-\sum_{j=1}^{K+1}N_j,
    \end{equation*}
    and by \eqref{Relazione} and \eqref{Utile}, it is immediate to see that
    \begin{equation}\label{Remainder}
        N_0\leq c_2\,c_3\,N^{(d-1)/(4d)}.
    \end{equation}
    Last, consider the constant depending on $u$ defined as
    \begin{equation*}
    {\rm M}(u)=\max_{|{n}|\geq 1}|{n}|^{-d} \int_{B_{|{n}|}}|\xi|^{d+1}|\hat{u}|^2(\xi)\de\xi,
    \end{equation*}
    and in particular, notice that this quantity is indeed finite. Finally, for $1\leq j \leq K+1$, we assign a uniform lattice $\mathcal{P}_{N_j}$ to every $N_j$ as in \eqref{Lattice}, and by proceeding as in the first part of the proof, we get the estimate
    \begin{equation}\label{Stima 1}
    \mathcal{D}_2(u;\mathcal{P}_{N_j})\leq \min\left(c_1 \mathcal{J}(u) N_j^{(d-1)/d},\,c_4 M(u) n_0^{d-1}\right),
    \end{equation}
    where the minimum involved depends on whether or not it holds $N_j^{1/d}\geq n_0$. On the other hand, for every set of $N_0$ points $\mathcal{P}_{N_0}\subset\mathbb{T}^d$, by the inequality in \eqref{Remainder}, we get
    \begin{equation}\label{Stima 2}
\begin{split}
    \mathcal{D}_2(u;\mathcal{P}_{N_0})&= \sum_{{n}\in\mathbb{Z}_*^d}\left|\sum_{{p}\in\mathcal{P}_{N_0}}e^{2 \pi i {p}\cdot{n}}\right|^2|{n}|^{-2d-1}\int_{B_{|{n}|}}|\xi|^{d+1}|\hat{u}|^2(\xi)\de\xi\\
    &\leq c_5 {\rm M}(u) N^{(d-1)/(2d)}.
\end{split}
    \end{equation}
    Therefore, by \eqref{Stima 1} and \eqref{Stima 2}, we get
    \begin{equation*}
    \begin{split}
        \mathcal{D}_2(u;\mathcal{P}_N)&= \sum_{{n}\in\mathbb{Z}_*^d}\left|\sum_{j=0}^{K+1}\sum_{{p}\in\mathcal{P}_{N_j}}e^{2 \pi i {p}\cdot{n}}\right|^2|{n}|^{-2d-1}\int_{B_{|{n}|}}|\xi|^{d+1}|\hat{u}|^2(\xi)\de\xi\\
        &\leq(K+2)\sum_{j=0}^{K+1}\sum_{{n}\in\mathbb{Z}_*^d}\left|\sum_{{p}\in\mathcal{P}_{N_j}}e^{2 \pi i {p}\cdot{n}}\right|^2|{n}|^{-2d-1}\int_{B_{|{n}|}}|\xi|^{d+1}|\hat{u}|^2(\xi)\de\xi\\
        &\leq (K+2){\rm M}(u)\left(c_4(K+1)n_0^{d-1}+c_5N^{(d-1)/(2d)}\right)+ c_1(K+2)\mathcal{J}(u)\sum_{j=1}^{K+1}N_j^{(d-1)/d}\\
        &\leq c_5\,{\rm M}(u)\,N^{(d-1)/(2d)}+ C_d\, \mathcal{J}(u)\,N^{(d-1)/d},
    \end{split}
    \end{equation*}
    so that, dividing by $N^{(d-1)/d}$ and by taking the $\limsup$ as $N\to+\infty$, the conclusion follows.
\end{proof}

Before proceeding with the proof of the lower bound, we need a technical lemma. We briefly present a result of Cassels~\cite{Cas56} and Montgomery~\cite{Mon94}. The following proof is inspired by Siegel's analytic proof of Minkowski's convex body theorem.
In particular, variations of such an argument proved suitable for tackling various questions of quadratic discrepancy, and we refer to \cite{BT22} and \cite{Ber24} for some recent developments. Moreover, we mention that an analogous argument holds on manifolds, as the interested reader may verify in \cite{BGG21}.

\begin{lemma}[Cassels-Montgomery]
	For every origin-symmetric convex body $\Omega\subset\mathbb{R}^d$, and for every finite set of points $\{{p}_j\}_{j=1}^N\subset\mathbb{T}^d$, it holds
	\begin{equation*}
	\sum_{{m}\in\Omega\cap\mathbb{Z}_*^d}\left|\sum_{j=1}^{N}e^{2\pi i {m}\cdot {p}_j}\right|^2\geq2^{-d}|\Omega|N- N^2.
	\end{equation*}
\end{lemma}

\begin{proof}
	Consider the auxiliary sets $A_\Omega({x})\subset\mathbb{Z}^d$ defined by
	\begin{equation*}
	A_\Omega({x})=\left({x}+\Omega/2\right)\cap \mathbb{Z}^d.
	\end{equation*}
	Notice that
	\begin{equation*}
	\int_{\mathbb{T}^d}\#A_\Omega({x})\de{x}=\int_{\mathbb{T}^d}\sum_{{n}\in\mathbb{Z}^d}\mathds{1}_{\Omega/2}({n}-{x})\de{x}=\int_{\mathbb{R}^d}\mathds{1}_{\Omega/2}({x})\de{x}=2^{-d}|\Omega|,
	\end{equation*}
	and therefore, we find a point
 \begin{equation*}
     \tilde{{x}}\in[0,1)^d\quad\text{such that}\quad\#A_\Omega(\tilde{{x}})\geq2^{-d}|\Omega|. 
 \end{equation*}
 Now, consider the non-negative trigonometric polynomial
	\begin{equation*}
	T({y})=\frac{1}{\#A_\Omega(\tilde{{x}})}\left|\sum_{{n}\in A_\Omega(\tilde{{x}})}e^{2\pi i {n}\cdot {y}}\right|^2=\frac{1}{\#A_\Omega(\tilde{{x}})}\sum_{{n},{m}\in A_\Omega(\tilde{{x}})}e^{2\pi i({n}-{m})\cdot {y}},
	\end{equation*}
	and notice that the function of its Fourier coefficients $\widehat T\colon\mathbb{Z}^d\to\mathbb{R}$ is non-negative as well and, since ${n},{m}\in A_\Omega(\tilde{{x}})$ imply $({n}-{m})\in\Omega$, then its support is contained in $\Omega$. Further, observe that we have
	\begin{equation*}
	T(0)=\#A_\Omega(\tilde{{x}})\geq 2^{-d}|\Omega|,
	\end{equation*}
	moreover, it is not difficult to notice that for all ${n}\in\mathbb{Z}^d$, it holds
	\begin{equation*}
	0\leq\widehat{T}({n})\leq\widehat{T}(\mathbf{0})=\int_{\mathbb{T}^d}T({x})\de x=1.
	\end{equation*}
	Therefore, it follows that
	\begin{equation*}
	\begin{split}
	\sum_{{n}\in\Omega\cap\mathbb{Z}^d}\left|\sum_{j=1}^{N}e^{2\pi i {n}\cdot {p}_j}\right|^2&\geq \sum_{{n}\in\Omega\cap\mathbb{Z}^d}\widehat{T}({n})\left|\sum_{j=1}^{N}e^{2\pi i {n}\cdot {p}_j}\right|^2\\
	&=\sum_{j=1}^{N}\sum_{\ell=1}^{N}\sum_{{n}\in\Omega\cap\mathbb{Z}^d}\widehat{T}({n})\,e^{2\pi i {n}\cdot({p}_j-{p}_\ell)}\\
	&=\sum_{j=1}^{N}\sum_{\ell=1}^{N}T({p}_j-{p}_\ell)\geq N2^{-d}|\Omega|,
	\end{split}
	\end{equation*}
 and the conclusion easily follows
\end{proof}

Finally, we use the previous lemma to prove the lower bound in Theorem~\ref{Lower}.

\begin{proof}[Proof of Theorem~\ref{Lower}]
We write $c_i=c_i(d)$ with $i\in\mathbb{N}$ to denote generic dimensional constants throughout the proof. By \eqref{Fundamental Asymptotic}, there exists $n_0\in\mathbb{N}$ be such that for every ${n}\geq n_0$, it holds
\begin{equation*}
    \int_{B_{|{n}|}}|\xi|^{d+1}|\hat{u}|^2(\xi)\de\xi\geq\frac{{|{n}|^{d}\, \mathcal{J}}(u)}{4d\pi^2}.
\end{equation*}
Hence, consider
\begin{equation*}
    {\rm m}(u)=\min_{1\leq|{n}|\leq n_0} |{n}|^{-2d-1} \int_{B_{|{n}|}}|\xi|^{d+1}|\hat{u}|^2(\xi)\de\xi.
\end{equation*}
First, we show that the claim holds for all $u$ such that ${\rm m}(u)>0$, and in particular, we notice that ${\rm m}(u)=0$ if and only if $\hat{u}$ is zero in $B_1$. We prove the lower bound for all the positive integers $N$ such that
\begin{equation*}
    \frac{ \mathcal{J}(u)}{4d\pi^2}N^{-(d+1)/d}\leq{\rm m}(u),
\end{equation*}
and for every set of $N$ points $\mathcal{P}_N\subset\mathbb{T}^d$. Indeed, within this assumption, for all $n\in B_{N^{1/d}}$ it holds
\begin{equation*}
    |{n}|^{-2d-1} \int_{B_{|{n}|}}|\xi|^{d+1}|\hat{u}|^2(\xi)\de\xi\geq \frac{ \mathcal{J}(u)}{4d\pi^2}N^{-(d+1)/d}.
\end{equation*}
Therefore, by applying the latter inequality in \eqref{Fourier Series}, we get that for a dimensional constant $\kappa\geq 1$ to be chosen later, it holds
\begin{equation*}
    \mathcal{D}_2(u;\mathcal{P}_N)\geq \sum_{{n}\in\mathbb{Z}_*^d\cap B_{(\kappa N)^{1/d}}}\left|\sum_{{p}\in\mathcal{P}_N}e^{2 \pi i {p}\cdot{n}}\right|^2 \frac{ \mathcal{J}(u)}{4 d \pi^2}(\kappa N)^{-(d+1)/d}.
\end{equation*}
Finally, we can apply the previous lemma and get
\begin{equation*}
    \mathcal{D}_2(u;\mathcal{P}_N)\geq c_1\, \mathcal{J}(u)\,(\kappa N)^{-(d+1)/d}\left( |B_1|\kappa N^2-N^2\right),
\end{equation*}
and by setting $\kappa=2|B_1|^{-1}$, it follows that
\begin{equation*}
   \mathcal{D}_2(u;\mathcal{P}_N) \geq c_2 \, \mathcal{J}(u)\,N^{(d-1)/d}.
\end{equation*}

Now, by contradiction, suppose there exists $u\in{\rm BV}(\mathbb{R}^d)\cap L^\infty(\R^d)$ be such that $\hat{u}$ is zero in $B_1$, and such that there exists an increasing sequence of integers $\{N_j\}_{j\in\mathbb{N}}$, with corresponding sets of $N_j$ points $\mathcal{P}_{N_j}\subset\mathbb{T}^d$, such that
\begin{equation}\label{Contraddizione}
    \mathcal{D}_2(u;\mathcal{P}_{N_j})\leq \frac{c_2}{4}\, \mathcal{J}(u)\,N_j^{(d-1)/d}.
\end{equation}
Now, notice that for a continuous function $v\in{\rm BV}(\mathbb{R}^d)\cap L^\infty(\R^d)$ such that $\hat{v}(\mathbf{0})>0$, the claim holds with $\mathcal{J}(v)=0$. Similarly, the claim would also hold for $u+v$ since
\begin{equation*}
    \mathcal{F}\{u+v\}(\mathbf{0})=\hat{v}(\mathbf{0})>0,
\end{equation*}
and therefore, since $\mathcal{J}(u+v)=\mathcal{J}(u)$, we would get
\begin{equation*}
    \liminf_{N\to+\infty} N^{(1-d)/d}\inf_{\mathcal{P}_N\subset\mathbb{T}^d}\mathcal{D}_2(u+v;\mathcal{P}_N)\geq c_2\,  \mathcal{J}(u).
\end{equation*}
Last, we notice that by the triangular inequality and by \eqref{Contraddizione}, for every $j\in\mathbb{N}$ it holds
\begin{equation*}
\begin{split}
    \mathcal{D}_2(u+v;\mathcal{P}_{N_j})&\leq 2\mathcal{D}_2(u;\mathcal{P}_{N_j}) + 2\mathcal{D}_2(v;\mathcal{P}_{N_j})\\
    &\leq\frac{c_2}{2}\, \mathcal{J}(u)\,N_j^{(d-1)/d}+o\big(N_j^{(d-1)/d}\big),
    \end{split}
\end{equation*}
and this is a contradiction. Therefore, the claim holds with $C_d=c_2/4$.
\end{proof}

\section{Appendix}\label{sec:Appendix}
For simplicity, in what follows we denote $\|\cdot\|_{L^p}$ with $\|\cdot\|_p$.
\begin{lemma}
    There exists a positive absolute constant $C$ such that, for every $u\in BV(\R^d)$ and for every $v\in\mathbb{S}^{d-1}$, it holds
    \begin{equation}
        \label{Boundedness of directional translation}
        \frac{1}{R}\int_{B_R}|\langle\xi,v\rangle|^2|\hat{u}|^2(\xi)d\xi\leq C \|u\|_{\infty}|D_vu|(\R^d)\quad\forall R>0.
    \end{equation}
    In particular, by integrating over $\mathbb{S}^{d-1}$, we get 
    \begin{equation}
        \label{Boundedness of translation}
        \frac{1}{R}\int_{B_R}|\xi|^2|\hat{u}|^2(\xi)d\xi\leq C \|u\|_{\infty}|Du|(\R^d)\quad\forall R>0.
    \end{equation}
\end{lemma}
\begin{proof}
    First, we fix a vector $v\in\mathbb{S}^{d-1}$ and a $h>0$. By standard properties of the Fourier transform and by Plancherel theorem, we have
    \begin{equation*}
        \int_{\R^d}\frac{|u(x+hv)-u(x)|^2}{h}\de x=\frac{2}{h}\int_{\R^d}(1-\cos(2\pi h\langle\xi, v\rangle))|\hat{u}|^2(\xi)d\xi.
    \end{equation*}
    The left-hand side is bounded from above by $2\|u\|_{\infty}|D_vu|(\R^d)$, where $|D_vu|$ stands for the directional total variation. Now, we exploit the fact that if $\xi\in B_{1/(2\pi h)}$, then there exists a positive absolute constant $c$ such that 
    \begin{equation*}
        ch^2|\langle\xi, v\rangle|^2\leq2(1-\cos(2\pi h\langle\xi, v\rangle)).
    \end{equation*}
    Therefore, we get
    \begin{equation*}
        h\int_{B_{1/h}}|\langle\xi, v\rangle|^2|\hat{u}|^2(\xi)d\xi\leq C\|u\|_{\infty}|D_vu|(\R^d),
    \end{equation*}
    which is \eqref{Boundedness of directional translation}.
    Finally, by setting $R=1/h$ and by integrating over $\mathbb{S}^{d-1}$, we obtain \eqref{Boundedness of translation}.
    
\end{proof}
\begin{rem}
By integrating by parts, we can reformulate \eqref{Boundedness of translation} as
\begin{equation*}
    \int_{B_R^c}|\hat{u}|^2(\xi)d\xi\leq C_d\frac{\|u\|_{\infty}|Du|(\R^d)}{R}.
\end{equation*}
\end{rem}

Last, we proceed with showing an interesting inequality. In particular, we obtain a non-sharp isoperimetric inequality by substituting $u=\mathds{1}_\Omega$ in the following proposition.

\begin{prop}
    There exists a positive dimensional constant $C_d$ such that, for every $u\in{\rm BV}(\R^d)\cap L^\infty(\R^d)$, It holds
    \begin{equation*}
         \|u\|_{2}^2\leq C_d\|u\|_{1}^{\frac{2}{d+1}}\|u\|_{\infty}^{\frac{d}{d+1}}\bigg(|Du|(\R^d)\bigg)^\frac{d}{d+1}.
    \end{equation*}
\end{prop}
\begin{proof}
First, we use Plancherel theorem to write
\begin{equation*}
\int_{\R^d}|u|^2(x)dx=\int_{B_R}|\hat{u}|^2(\xi)d\xi+\int_{B_R^c}|\hat{u}|^2(\xi)d\xi.
\end{equation*}
By the fact that the Fourier transform is $1$-Lipschitz between $L^1$ and $L^\infty$, and by \eqref{Boundedness of translation}, we obtain
\begin{equation*}
    \|u\|_{2}^2\leq \omega_dR^d\|u\|_{1}^2+C\frac{\|u\|_{\infty}|Du|(\R^d)}{R}.
\end{equation*}
By optimizing over $R>0$, we get the optimal value
\begin{equation*}
    R = \bigg(\frac{C\|u\|_{\infty}|Du|(\R^d)}{d\omega_d\|u\|_{1}^2}\bigg)^\frac{1}{d+1}.
\end{equation*}
Therefore, it follows that
\begin{equation*}
     \|u\|_{2}^2\leq C_d\|u\|_{1}^{\frac{2}{d+1}}\|u\|_{\infty}^{\frac{d}{d+1}}\bigg(|Du|(\R^d)\bigg)^\frac{d}{d+1},
\end{equation*}
and this is the thesis.
\end{proof}

\newpage

\bibliography{references}
            \bibliographystyle{alpha}

\end{document}